\documentclass{article}
\usepackage{amsfonts}
\usepackage{latexsym}
\usepackage{mathrsfs}
\usepackage{amssymb}
\usepackage{amsmath}
\usepackage{amsthm}
\usepackage{indentfirst}
\usepackage{color}
\usepackage{graphicx}

\allowdisplaybreaks
\newtheorem{theorem}{\color{black}\indent Theorem}[section]
\newtheorem{lemma}{\color{black}\indent Lemma}[section]
\newtheorem{proposition}{\color{black}\indent Proposition}[section]
\newtheorem{definition}{\color{black}\indent Definition}[section]
\newtheorem{remark}{\color{black}\indent Remark}[section]
\newtheorem{corollary}{\color{black}\indent Corollary}[section]

\begin{document}
\title{\LARGE\bf Extinction of solutions to a class of fast diffusion systems with nonlinear sources}
\author{Yuzhu Han$^\dag$ \qquad Wenjie Gao}
 \date{}
 \maketitle

 \footnotetext{\hspace{-1.9mm}$^\dag$Corresponding author.\\
 Email addresses: yzhan@jlu.edu.cn(Y. Han),  wjgao@jlu.edu.cn(W. Gao).

\thanks{
 $^*$The project is supported by NSFC (11271154), by Key Lab of Symbolic
 Computation and Knowledge Engineering of Ministry of Education and
 by the 985 program of Jilin University.}}

\begin{center}
{\it\small School of Mathematics, Jilin University,
 Changchun 130012, P.R. China}
\end{center}

\date{}
\maketitle

{\bf Abstract}\ In this paper, the finite time extinction of solutions to the fast diffusion system
$u_t=\mathrm{div}(|\nabla u|^{p-2}\nabla u)+v^m$, $v_t=\mathrm{div}(|\nabla v|^{q-2}\nabla v)+u^n$
is investigated, where $1<p,q<2$, $m,n>0$ and $\Omega\subset \mathbb{R}^N\ (N\geq1)$ is a bounded smooth domain.
After establishing the local existence of weak solutions, the authors show that if $mn>(p-1)(q-1)$, then any solution
vanishes in finite time provided that the initial data are ``comparable";
if $mn=(p-1)(q-1)$ and $\Omega$ is suitably small, then the existence of extinction solutions
for small initial data is proved by using the De Giorgi iteration process and comparison method.
On the other hand, for $1<p=q<2$ and $mn<(p-1)^2$, the existence of at least one non-extinction solution for
any positive smooth initial data is proved.



\section{Introduction}
\setcounter{equation}{0}

This paper is concerned with the extinction properties of solutions to the following fast diffusion
parabolic system
\begin{equation}\label{1.1}
\begin{cases}
u_t=\mathrm{div}(|\nabla u|^{p-2}\nabla u)+v^m,&\ x\in\Omega,\ t>0,\\
v_t=\mathrm{div}(|\nabla v|^{q-2}\nabla v)+u^n,&\ x\in\Omega,\ t>0,\\
u(x,t)=v(x,t)=0,&x\in\partial\Omega,\ t>0,\\
u(x,0)=u_0(x),\ v(x,0)=v_0(x),&x\in\Omega,
\end{cases}
\end{equation}
where $1<p,\ q<2$, $m,\ n>0$, $\Omega$ is a bounded domain in
$\mathbb{R}^N$ ($N\geq1$) with smooth boundary $\partial\Omega$ and the initial data $u_0\in L^\infty({\Omega})\cap W^{1,p}_0(\Omega)$,
$v_0\in L^\infty({\Omega})\cap W^{1,q}_0(\Omega)$.

Problem (\ref{1.1}) appears, for example, in the theory of non-Newtonian filtration fluids \cite{Dibenedetto93,Wu01}.
From a physical point of view, we need only to consider the nonnegative solutions.
Moreover, if we assume that $u_0(x)$ and $v_0(x)$ are nonnegative, then we can deduce, by the weak maximum
principle, that $u$ and $v$ are nonnegative as long as they exist. Therefore, we always assume that the initial data
are nonnegative nontrivial functions and consider only the nonnegative solutions throughout this paper.

In this paper, we are interested in the extinction in finite time of solutions to (\ref{1.1}).
We say that a solution $(u,v)$ has a finite extinction time $T$ if $T>0$ is the smallest
number such that both $u(x,t)=0$ and
$v(x,t)=0$ for a.e. $(x,t)\in\Omega\times(T,\infty)$.

Finite time extinction is one of the most important properties of solutions to many evolutionary equations
that has been investigated by many authors during the
past several decades. It is E. Sabinina who first observed extinction via fast
diffusion \cite{Sabinina62}, and from
then on, there has been increasing interest in this direction. For example, in his
fundamental survey \cite{Kalashnikov74},  A. S. Kalashnikov investigated finite time
extinction as well as localization and finite propagation properties of solutions
to the following semilinear heat equation with homogeneous Dirichlet
boundary condition
\begin{equation}\label{1.2}
\begin{cases}
u_t=\Delta u-u^q,&x\in\Omega,\ t>0,\\
u(x,t)=0,&x\in\partial\Omega,\ t>0,\\
u(x,0)=u_0(x),&x\in\Omega
\end{cases}
\end{equation}
in the 1970s. A more complete extinction conclusion of Problem (\ref{1.2}) was given in
\cite{Gu94}: A nontrivial solution of (\ref{1.2}) vanishes in finite
time if and only if $0<q<1$, which means that strong absorption will
cause extinction to occur in finite time. In \cite{Gu94}, Gu also gave a simple statement of the necessary and
sufficient conditions of extinction of the solution to the following
problem
\begin{equation}\label{1.3}
\begin{cases}
u_t=\mathrm{div}(|\nabla u|^{p-2}\nabla u)+au^q,&x\in\Omega,\ t>0,\\
u(x,t)=0,&x\in\partial\Omega,\ t>0,\\
u(x,0)=u_0(x),&x\in\Omega,
\end{cases}
\end{equation}
with $a<0,q>0$. He proved that if $p\in(1,2)$ or $q\in(0,1)$ the
solutions of the problem vanish in finite time, but if $p\geq2$ and
$q\geq1$, there is no extinction. In the absence of absorption (i.e.
$a=0$), Dibenedetto \cite{Dibenedetto93} and Yuan et al. \cite{Yuan05}
proved that the necessary and sufficient conditions for the
extinction to occur is $p\in(1,2)$.

Later in \cite{Yin07}, Yin and Jin studied Problem (\ref{1.3})
with $1<p<2$, $a,q>0$ and dimension $N>2$. They proved that if $q>p-1$, then any
bounded and non-negative weak solution of Problem (\ref{1.3})
vanishes in finite time for appropriately small initial data $u_0$,
while Problem (\ref{1.3}) admits at least one bounded non-negative
and non-extinction weak solution for the case of $0<q<p-1$. As for
the critical case $q=p-1$, whether the solutions vanish in finite
time or not depends on the comparison between $a$ and $\lambda_1$,
where $\lambda_1>0$ is the first eigenvalue of $p$-Laplace operator in $\Omega$ with homogeneous Dirichlet
boundary conditions. Extinction and non-extinction results similar to the ones in
\cite{Yin07} were also obtained by Tian and Mu in \cite{Tian08}, and
some sufficient conditions in \cite{Yin07} for the solutions of
(\ref{1.3}) to vanish in finite time were weakened by Liu and Wu
(see \cite{Liu08}). There are some other extinction results of the solutions of
degenerate or singular parabolic problems with or without absorption
(reaction) terms, readers may refer to
\cite{Berryman80,Ferreira01,Friedman87,Friedman80,Galaktionov00,Galaktionov91,Galaktionov94a,Galaktionov94b,Han11,Han13,Jin09,Yin09}
and references therein.

Generally speaking, for Problems (\ref{1.2}) and (\ref{1.3}) with $a<0$, there
is a cooperation between the diffusion term and the absorption term,
and fast diffusion or strong absorption might cause any bounded
nonnegative solution to vanish in finite time. However, in
(\ref{1.3}) with $a>0$, the nonlinear term is  physically called the ``hot
source", while in (\ref{1.2}) and (\ref{1.3}) with $a<0$ the nonlinear term
is usually called the ``cool source". Results in
\cite{Li05,Tian08,Yin07} imply that when the diffusion is fast enough,
the solutions might still vanish in finite time for small initial
data in spite of the ``hot sources".

However, compared with the huge amount of extinction results concerning scalar problems,
there is only quite little literature dealing with
extinction quality of solutions to evolutionary systems until now.
In \cite{Friedman92}, Friedman et al. investigated the extinction
and positivity for the following system of semilinear parabolic
variational inequalities
\begin{equation}\label{1.4}
\begin{cases}
u_t-u_{xx}+v^p\geq0,\ \ \ \ v_t-v_{xx}+u^q\geq0,&(x,t)\in (-1,1)\times(0,\infty),\\
u(u_t-u_{xx}+v^p)=0,\ v(v_t-v_{xx}+u^q)=0,&(x,t)\in (-1,1)\times(0,\infty),\\
u\geq0,\ \ v\geq0, &(x,t)\in (-1,1)\times(0,\infty),\\
u(\pm1,t)=v(\pm1,t)=0,&t\in(0,\infty),\\
u(x,0)=u_0(x),\ v(x,0)=v_0(x), &x\in [-1,1].
\end{cases}
\end{equation}
It was shown that when $u_0$ and $v_0$ are ``comparable", then at least one of the components becomes
extinct in finite time provided that $pq<1$. On the other hand, for any $p=q>0$, there are initial values
for which neither $u$ nor $v$ vanishes in any finite time.

In a quite recent paper \cite{Chen-Y13}, Chen et al. studied the following fast diffusion system
\begin{equation}\label{1.5}
\begin{cases}
u_t=\Delta u^m+v^p,&x\in\Omega,\ t>0,\\
v_t=\Delta v^n+u^q,&x\in\Omega,\ t>0,\\
u(x,t)=v(x,t)=0,&x\in\partial\Omega,\ t>0,\\
u(x,0)=u_0(x),\ v(x,0)=v_0(x),&x\in\Omega,
\end{cases}
\end{equation}
where $0<m,n<1$, $p,q>0$ and $\Omega\subset \mathbb{R}^N(N>2)$ is a bounded
domain with smooth boundary $\partial\Omega$. It was proved that
if $pq>mn$ and the initial data are ``comparable" in some sense,
then any solution of (\ref{1.5}) vanishes in finite time; if $pq=mn$ and $\lambda_1$
(the first eigenvalue of $-\Delta$ in $\Omega$ with homogeneous boundary condition)
is large enough, then there exists a solution vanishing in finite time for small initial data.
However, they did not show whether there exists no-extinction solution or not when $pq<mn$.

Motivated by the works mentioned above, we shall study the extinction properties of
solutions to (\ref{1.1}) for any $N\geq1$ and give some conditions for the solutions to vanish in finite time,
extending some results obtained in \cite{Chen-Y13,Tian08,Yin07} to system (\ref{1.1}). However,
we encounter two difficulties when doing so. The first one
is that the nonlinearities in (\ref{1.1}) may be non-Lipschitz, which excludes the possibility of
applying the general comparison principles to (\ref{1.1}) and the uniqueness is also false in general,
and the second one is that we find it hard to construct a suitable supersolution which vanishes 
in finite time for the case $mn>(p-1)(q-1)$. To overcome these difficulties and to give some sufficient conditions for the solutions to vanish in
finite time, we first establish a weak form comparison principle (which requires that the supersolution
has a positive lower bound in the domain), and then, by referring to a lemma describing the
invariant region of a specially constructed ordinary differential system and by modifying the
integral estimates methods used in \cite{Chen-Y13}, we show that the solutions of (\ref{1.1}) vanish
in finite time when the nonlinear sources are in some sense weak and when the initial data $u_0$ and
$v_0$ are ``comparable". Furthermore, we obtain a non-extinction result for some special cases, which,
to the best of our knowledge, seems to be first work concerning the non-extinction results of quasilinear
parabolic systems with sources. It is worth mentioning that our methods can not only be used to deal with problems
for the equations in (\ref{1.1}) with local or nonlocal sources, but can also be applied to treat the problem in \cite{Chen-Y13}
with a simplified proof. Moreover, the cases $N=1,2$ can also be included.

The rest of this paper is organized as follows. In Section 2, we
introduce the definition of weak solutions, prove a weak comparison principle and establish the local
existence of weak solutions. The proofs of the main results will be presented in Section 3.

\par
\section{Preliminaries}
\setcounter{equation}{0}

In this section, as preliminaries, we introduce some definitions and
notations. It is well known that the equations in (\ref{1.1}) are
singular when $1<p,q<2$, and hence there is no classical solution in general.
Therefore, we have to consider its solutions in some weak sense. We first introduce some
notations which will be used throughout this paper. For any $T\in(0,\infty)$ and $0<t_1<t_2<\infty$,
we denote $Q_T=\Omega\times(0,T)$, $\Gamma_T=\partial\Omega\times(0,T)$ and
\begin{equation*}
\begin{split}
&Q=\Omega\times(0,\infty),\ \ Q_{(t_1,t_2)}=\Omega\times(t_1,t_2),\\
&E_n=\Big\{w\in L^{2n}(Q_T)\cap L^2(Q_T);\frac{\partial w}{\partial
t}\in L^2(Q_T), \nabla w \in L^p(Q_T)\Big\},\\
&E_m=\Big\{w\in L^{2m}(Q_T)\cap L^2(Q_T);\frac{\partial w}{\partial
t}\in L^2(Q_T), \nabla w \in L^q(Q_T)\Big\},\\
&E_p=\Big\{w\in L^2(Q_T);\nabla w\in L^p(Q_T)\Big\},\ E_q=\Big\{w\in L^2(Q_T);\nabla w\in L^q(Q_T)\Big\},\\
&E_{p0}=\{w\in E_p;\ w\mid_{\partial\Omega}=0\},\ E_{q0}=\{w\in
E_q;\ w\mid_{\partial\Omega}=0\}.
\end{split}
\end{equation*}

\begin{definition}\label{weak-solutions}
A nonnegative vector valued function $(u,v)$ with $u\in E_n$ and $v\in E_m$ is called a nonnegative subsolution of (\ref{1.1}) in $Q_T$ provided that
for any $0\leq\phi_1\in E_{p0}$ and $0\leq\phi_2\in E_{q0}$
\begin{equation*}
\begin{cases}
\iint_{Q_T}\Big(\frac{\partial u}{\partial t}\phi_1+|\nabla u|^{p-2}\nabla u\nabla\phi_1\Big)dxd\tau\leq\iint_{Q_T}v^m\phi_1dxd\tau,\\
\iint_{Q_T}\Big(\frac{\partial v}{\partial t}\phi_2+|\nabla v|^{q-2}\nabla v\nabla\phi_2\Big)dxd\tau\leq\iint_{Q_T}u^n\phi_2dxd\tau,\\
u(x,t)\leq0,\ \ v(x,t)\leq0,&x\in\Gamma_T,\\
u(x,0)\leq u_0(x),\ \ v(x,0)\leq v_0(x), &x\in\Omega.
\end{cases}
\end{equation*}
By replacing $\leq$ by $\geq$ in the above inequalities we obtain the definition of weak supersolutions of (\ref{1.1}).
Furthermore, if $(u,v)$ is a weak supersolution as well as a weak subsolution solution, then we call it a weak solution
of Problem (\ref{1.1}).
\end{definition}

In order to prove the main results of this paper, the following weak comparison principle
is needed.

\begin{lemma}\label{comparison}
Let $(\overline{u},\overline{v})$ and $(\underline{u},\underline{v})$ be a pair of bounded weak super and
sub-solution of Problem (\ref{1.1}) in $Q_T$, and there exists a constant $\delta>0$ such
that $(\overline{u},\overline{v})\geq (\delta,\delta)$.
Then $(\overline{u},\overline{v})\geq(\underline{u},\underline{v})$ a.e. in $Q_T$. Moreover, if $m,n\geq1$,
the condition $(\overline{u},\overline{v})\geq (\delta,\delta)$ is unnecessary.
\end{lemma}

\begin{proof}
The proof is more or less standard. However, for completeness, we prefer to sketch the outline here.
From the definition of weak super and subsolutions, we obtain, for any $0\leq\phi_1\in E_{p0}$ and $0\leq\phi_2\in E_{q0}$,
\begin{eqnarray*}
&&\iint_{Q_T}\Big(\frac{\partial \underline{u}}{\partial t}-\frac{\partial \overline{u}}{\partial t}\Big)\phi_1dxd\tau
+\iint_{Q_T}(|\nabla \underline{u}|^{p-2}\nabla \underline{u}-|\nabla \overline{u}|^{p-2}\nabla \overline{u})\nabla\phi_1dxd\tau\\
&&\leq\iint_{Q_T}(\underline{v}^m-\overline{v}^m)\phi_1dxd\tau,\\
&&\iint_{Q_T}\Big(\frac{\partial \underline{v}}{\partial t}-\frac{\partial \overline{v}}{\partial t}\Big)\phi_2dxd\tau
+\iint_{Q_T}(|\nabla \underline{v}|^{q-2}\nabla \underline{v}-|\nabla \overline{v}|^{q-2}\nabla \overline{v})\nabla\phi_2dxd\tau\\
&&\leq\iint_{Q_T}(\underline{u}^n-\overline{u}^n)\phi_2dxd\tau.
\end{eqnarray*}
We first prove the conclusion when $m,n\geq1$. Denote $M=\max\Big\{\|\overline{u}\|_{L^\infty(Q_T)},\|\overline{v}\|_{L^\infty(Q_T)},\|\underline{u}\|_{L^\infty(Q_T)}$, $\|\underline{v}\|_{L^\infty(Q_T)}\Big\}$.
For any $t\in(0,T)$, by choosing $\phi_1=\chi_{[0,t]}(\underline{u}-\overline{u})_+$, $\phi_2=\chi_{[0,t]}(\underline{v}-\overline{v})_+$,
we have
\begin{eqnarray*}
&&\iint_{Q_t}\Big(\frac{\partial \underline{u}}{\partial t}-\frac{\partial \overline{u}}{\partial t}\Big)(\underline{u}-\overline{u})_+dxd\tau
+\iint_{Q_t}(|\nabla \underline{u}|^{p-2}\nabla \underline{u}-|\nabla \overline{u}|^{p-2}\nabla \overline{u})\nabla(\underline{u}-\overline{u})_+dxd\tau\\
&&\leq\iint_{Q_t}(\underline{v}^m-\overline{v}^m)(\underline{u}-\overline{u})_+dxd\tau,\\
&&\leq mM^{m-1}\iint_{Q_t}(\underline{v}-\overline{v})_+(\underline{u}-\overline{u})_+dxd\tau,
\end{eqnarray*}
where $\chi_{[0,t]}$ is the characteristic function defined on $[0,t]$ and $s_+=\max\{s,0\}$.
By a direct computation, we arrive at
\begin{eqnarray}\label{2.1}
&&\int_{\Omega}(\underline{u}-\overline{u})^2_+dx+2\iint_{Q_t}(|\nabla \underline{u}|^{p-2}\nabla \underline{u}-|\nabla \overline{u}|^{p-2}\nabla \overline{u})\nabla(\underline{u}-\overline{u})_+dxd\tau\nonumber\\
&\leq& 2mM^{m-1}\iint_{Q_t}(\underline{v}-\overline{v})_+(\underline{u}-\overline{u})_+dxd\tau.
\end{eqnarray}
Symmetrically, we have
\begin{eqnarray}\label{2.2}
&&\int_{\Omega}(\underline{v}-\overline{v})^2_+dx+2\iint_{Q_t}(|\nabla \underline{v}|^{q-2}\nabla \underline{v}-|\nabla \overline{v}|^{q-2}\nabla \overline{v})\nabla(\underline{v}-\overline{v})_+dxd\tau\nonumber\\
&\leq& 2nM^{n-1}\iint_{Q_t}(\underline{v}-\overline{v})_+(\underline{u}-\overline{u})_+dxd\tau.
\end{eqnarray}
Recalling the monotonicity of $p$-Laplace operator and Gronwall's inequality one has
$$\int_\Omega[(\underline{u}-\overline{u})^2_++(\underline{v}-\overline{v})^2_+]dx\leq0,$$
which implies that $(\overline{u},\overline{v})\geq(\underline{u},\underline{v})$.
The proof of the other cases is much the same as above only with the exception that the coefficients
on the right hand side of (\ref{2.1}) and (\ref{2.2}) may depend on $\delta$. We omit the details
and the proof is complete.
\end{proof}

\begin{proposition}\label{existence}
Assume that $0\leq u_0(x)\in L^{\infty}(\Omega)\cap
W^{1,p}_0(\Omega)$ and $0\leq v_0(x)\in L^{\infty}(\Omega)\cap
W^{1,q}_0(\Omega)$. Then there exists a $T=T(u_0,v_0)>0$
such that Problem (\ref{1.1}) admits at least one bounded and
nonnegative weak solution $(u,v)$ in the cylinder $Q_T$; Furthermore,
if $m,n\geq1$, then the weak solution is unique.
\end{proposition}

\begin{proof}
Consider the following auxiliary problem
\begin{equation}\label{regularization}
\begin{cases}
u_{kt}=\mathrm{div}((|\nabla u_k|^2+\varepsilon_k)^\frac{p-2}{2}\nabla u_k)+v_k^m,&x\in\Omega,\ t>0,\\
v_{kt}=\mathrm{div}((|\nabla v_k|^2+\delta_k)^\frac{q-2}{2}\nabla v_k)+u_k^n,&x\in\Omega,\ t>0,\\
u(x,t)=v(x,t)=0,&x\in\partial\Omega,\ t>0,\\
u(x,0)=u^{\varepsilon_k}_0(x),\ v(x,0)=v^{\delta_k}_0(x),&x\in\Omega,
\end{cases}
\end{equation}
where $\{\varepsilon_k\}$, $\{\delta_k\}$ are strictly decreasing sequences, $0<\varepsilon_k,\delta_k<1$,
and $\varepsilon_k,\ \delta_k\rightarrow 0$ as $k\rightarrow\infty$. $u^{\varepsilon_k}_0\in C^\infty_0(\overline{\Omega})$
and $v^{\delta_k}_0\in C^\infty_0(\overline{\Omega})$ are approximation functions of the initial data $u_0(x)$ and $v_0(x)$, respectively.
$\|u^{\varepsilon_k}_0\|_{L^\infty(\Omega)}\leq\|u_0\|_{L^\infty(\Omega)}$,
$\|\nabla u^{\varepsilon_k}_0\|_{L^p(\Omega)}\leq C_0\|\nabla u_0\|_{L^p(\Omega)}$ for all $\varepsilon_k$, and
$u^{\varepsilon_k}_0\rightarrow u_0$ strongly in $W^{1,p}_0(\Omega)$; $\|v^{\delta_k}_0\|_{L^\infty(\Omega)}\leq\|v_0\|_{L^\infty(\Omega)}$,
$\|\nabla v^{\delta_k}_0\|_{L^q(\Omega)}\leq C_0\|\nabla v_0\|_{L^q(\Omega)}$ for all $\delta_k$, and
$v^{\delta_k}_0\rightarrow v_0$ strongly in $W^{1,q}_0(\Omega)$. Here $C_0>0$ is a constant independent of $k$.

Since (\ref{regularization}) is a nondegenerate problem for each fixed $\varepsilon^k$ and $\delta^k$,
it is easy to prove that it admits a unique classical solution $(u_k, v_k)$ by using the Schauder's
fixed point theorem. Moreover, by the weak maximum principle we know that $u_k, v_k\geq0$ for each $k$.
To find the limit functions of $(u_k,v_k)$, we need to derive some uniform estimates.
The whole process will be divided into four steps.

{\bf Step 1.} There exist a small constant $T_0>0$ and a positive constant $M_1$, independent of $k$, such that
\begin{equation}\label{upper bound}
\|u_k\|_{L^\infty(Q_{T_0})},\ \|v_k\|_{L^\infty(Q_{T_0})}\leq M_1.
\end{equation}
To this end, we only need to consider the following Cauchy problem
\begin{equation}\label{ODE}
\begin{split}
&\frac{dU}{dt}=V^m,\qquad\frac{dV}{dt}=U^n,\ \ t>0,\\
&U(0)=\|u_{0}\|_{L^{\infty}(\Omega)},~~V(0)=\|v_{0}\|_{L^{\infty}(\Omega)}.
\end{split}
\end{equation}
It is known from the theories in ODEs that there exists a constant $t_0>0$
depending only on $\|u_{0}\|_{L^{\infty}(\Omega)}$ and $\|v_{0}\|_{L^{\infty}(\Omega)}$ such that Problem (\ref{ODE})
admits a solution $(U,V)$ on $[0,t_0]$. Moreover, $(U,V)$ is increasing. By the comparison principle for uniformly parabolic equations (see \cite{Potter67})
we know that $((u_k,v_k))\leq (U,V)$ as long as they exist. Set $T_0=\frac{t_0}{2}$ and $M_1=\max\{U(T_0),V(T_0)\}$, then
(\ref{upper bound}) follows.

{\bf Step 2.} There exists a constant $M_2>0$, independent of $k$, such that
\begin{equation}\label{bound of gradient}
\|\nabla u_k\|_{L^p(Q_{T_0})},\ \ \|\nabla v_k\|_{L^q(Q_{T_0})}\leq M_2.
\end{equation}
Multiplying the first equation in (\ref{regularization}) by $u_k$ and integrating the results over $Q_{T_0}$, we obtain
\begin{eqnarray*}
&&\frac{1}{2}\int_\Omega u_k^2(x,T_0)dx+\iint_{Q_{T_0}}(|\nabla u_k|^2+\varepsilon_k)^\frac{p-2}{2}|\nabla u_k|^2dxdt\\
&=&\iint_{Q_{T_0}}v_k^m u_k dxdt+\frac{1}{2}\int_\Omega \Big(u^{\varepsilon_k}_0(x)\Big)^2dx.
\end{eqnarray*}
By combining the fact $\|u^{\varepsilon_k}_0\|_{L^\infty(\Omega)}\leq\|u_0\|_{L^\infty(\Omega)}$ with (\ref{upper bound})
we have
\begin{equation*}
\iint_{Q_{T_0}}(|\nabla u_k|^2+\varepsilon_k)^\frac{p-2}{2}|\nabla u_k|^2dxdt\leq C,
\end{equation*}
where $C$ is a positive constant that does not depend on $k$. Note that
\begin{eqnarray*}
&&\iint_{Q_{T_0}}|\nabla u_k|^pdxdt\leq\iint_{Q_{T_0}}(|\nabla u_k|^2+\varepsilon_k)^\frac{p}{2}dxdt\\
&=&\iint_{Q_{T_0}}(|\nabla u_k|^2+\varepsilon_k)^\frac{p-2}{2}|\nabla u_k|^2dxdt+\varepsilon_k\iint_{Q_{T_0}}(|\nabla u_k|^2+\varepsilon_k)^\frac{p-2}{2}dxdt.
\end{eqnarray*}
To prove the boundedness of $\|\nabla u_k\|_{L^p(Q_{T_0})}$, it suffices to estimate the upper bound of
\begin{equation*}
I=\varepsilon_k\iint_{Q_{T_0}}(|\nabla u_k|^2+\varepsilon_k)^\frac{p-2}{2}dxdt.
\end{equation*}
Since $1<p<2$, it follows from $0<\varepsilon_k<1$ that
\begin{equation*}
I=\iint_{Q_{T_0}}\Big(\frac{\varepsilon_k}{|\nabla u_k|^2+\varepsilon_k}\Big)^{\frac{2-p}{2}}\varepsilon_k^{\frac{p}{2}}dxdt
\leq\iint_{Q_{T_0}}\varepsilon_k^{\frac{p}{2}}dxdt\leq|Q_{T_0}|.
\end{equation*}
By applying similar arguments we can prove that $\|\nabla v_k\|_{L^q(Q_{T_0})}$ is also bounded uniformly in $k$.
Therefore, (\ref{bound of gradient}) is valid.

{\bf Step 3.} There exists a constant $M_3>0$, independent of $k$, such that
\begin{equation}\label{bound of time derivitive}
\|u_{kt}\|_{L^2(Q_{T_0})},\ \ \|v_{kt}\|_{L^2(Q_{T_0})}\leq M_3.
\end{equation}
To do so, multiplying the first equation in (\ref{regularization}) by $u_{kt}$ and integrating the results over $Q_{T_0}$, one has
\begin{eqnarray*}
&&\iint_{Q_{T_0}}u^2_{kt}dxdt+\iint_{Q_{T_0}}(|\nabla u_k|^2+\varepsilon_k)^\frac{p-2}{2}\nabla u_k\nabla u_{kt}dxdt\\
&=&\iint_{Q_{T_0}}v_k^mu_{kt}dxdt.
\end{eqnarray*}
By using Cauchy's inequality and the equality
\begin{eqnarray*}
&&\iint_{Q_{T_0}}(|\nabla u_k|^2+\varepsilon_k)^\frac{p-2}{2}\nabla u_k\nabla u_{kt}dxdt\\
&=&\frac{1}{p}\int_\Omega(|\nabla u_k(x,T_0)|^2+\varepsilon_k)^\frac{p}{2}dx-\frac{1}{p}\int_\Omega(|\nabla u^{\varepsilon_k}_0|^2+\varepsilon_k)^\frac{p}{2}dx,
\end{eqnarray*}
we deduce that
\begin{eqnarray*}
\iint_{Q_{T_0}}u^2_{kt}dxdt&\leq&-\frac{1}{p}\int_\Omega(|\nabla u_k(x,T_0)|^2+\varepsilon_k)^\frac{p}{2}dx+\frac{1}{p}\int_\Omega(|\nabla u^{\varepsilon_k}_0|^2+\varepsilon_k)^\frac{p}{2}dx\\
&+&\frac{1}{2}\iint_{Q_{T_0}}v_k^{2m}dxdt+\frac{1}{2}\iint_{Q_{T_0}}u^2_{kt}dxdt,
\end{eqnarray*}
which implies that
\begin{equation}\label{2.8}
\iint_{Q_{T_0}}u^2_{kt}dxdt\leq\iint_{Q_{T_0}}v_k^{2m}dxdt+\frac{2}{p}\int_\Omega(|\nabla u^{\varepsilon_k}_0|^2+\varepsilon_k)^\frac{p}{2}dx.
\end{equation}
Noticing $1<p<2$, recalling $\|\nabla u^{\varepsilon_k}_0\|_{L^p(\Omega)}\leq C_0\|\nabla u_0\|_{L^p(\Omega)}$
and the basic inequality
$$(a+b)^r\leq a^r+b^r,\ \ a,b>0, r\geq1,$$
we conclude that
\begin{equation*}
\int_\Omega(|\nabla u^{\varepsilon_k}_0|^2+\varepsilon_k)^\frac{p}{2}dx\leq\int_\Omega(|\nabla u^{\varepsilon_k}_0|^2+1)^\frac{p}{2}dx
\leq\int_\Omega(|\nabla u^{\varepsilon_k}_0|^p+1)dx\leq M_3,
\end{equation*}
which together with (\ref{2.8}) guarantees the boundedness of $\|u_{kt}\|_{L^2(Q_{T_0})}$.
The upper bound of $\|v_{kt}\|_{L^2(Q_{T_0})}$ can be derived similarly.

Inequalities (\ref{upper bound}), (\ref{bound of gradient}) and (\ref{bound of time derivitive}) imply that
there exists a subsequence of $(u_k,v_k)$, still denoted by $(u_k,v_k)$ such that
\begin{align}
&u_{k}\rightarrow u,~v_{k}\rightarrow v,~~ for~~a.e.~~(x,t)\in~Q_{T_{0}},\\
&\nabla u_{k}\rightharpoonup \nabla u,~~~~in~~
L^{p}(0,T_{0};L^{p}(\Omega)),\\
&\nabla v_{k}\rightharpoonup \nabla v,~~~~~in~~
L^{q}(0,T_{0};L^{q}(\Omega)),\\
&u_{kt}\rightharpoonup u_{t},~v_{kt}\rightharpoonup v_{t},~~~~in~~
L^{2}(0,T_{0};L^{2}(\Omega)),\\
&|\nabla
u_{k}|^{p-2}(u_{k})_{x_{i}}\rightharpoonup\omega_{i},~~~~in~~
L^{\frac{p}{p-1}}(0,T_{0};L^{\frac{p}{p-1}}(\Omega)),\\
&|\nabla v_{k}|^{q-2}(u_{k})_{x_{i}}\rightharpoonup z_{i},~~~~in~~
L^{\frac{q}{q-1}}(0,T_{0};L^{\frac{q}{q-1}}(\Omega)),
\end{align}
where ``$\rightharpoonup$" denotes weak convergence in the corresponding Banach spaces.

{\bf Step 4.} We show that $|\nabla u|^{p-2}u_{x_{i}}=\omega_{i}$ and $|\nabla v|^{q-2}v_{x_{i}}=z_{i}$.

This can be done by choosing $\phi_1=\Phi_1(u_k-u)$ and $\phi_2=\Phi_2(v_k-v)$ as the test functions
with nonnegative functions $\Phi_1,\Phi_2\in C^{1,1}(Q_{T_0})$ and by using the same trick as that in \cite{Zhao93}.
We omit the details.

Thus, the proof of the local existence of weak solutions is complete by a standard limiting process.
The uniqueness of the solution with $m,n\geq1$ is a direct corollary of Lemma~\ref{comparison}.
The proof is complete.
\end{proof}

\par
\section{Proofs of the main results}
\setcounter{equation}{0}

In this section, by using the method of comparison
principle and integral estimates, we shall prove our main results
and give some sufficient conditions for the solutions of (\ref{1.1})
to vanish in finite time. The following two lemmas, which describe the invariant region
of an ordinary differential system, will play important roles in the forthcoming
proofs.

\begin{lemma}\label{le3.1}
\cite{Chen-Y13}\ Let $a_{i}, b_{i}(i=1,2), m, n$ be positive constants, $1<p,q <2$ and $mn\geq(p-1)(q-1)$. Denote
$$\mathcal{Q}=\Big\{(W_1,W_2)\in \mathbb{R}^2|W_1\geq0, W_2\geq0\ and\ \Big(\frac{b_1}{\delta a_1}\Big)^{\frac{1}{p-1}}W_2^{\frac{m}{p-1}}\leq W_1\leq\Big(\frac{\delta a_2}{b_2}\Big)^{\frac{1}{n}}W_2^{\frac{q-1}{n}}\Big\},$$
where $0<\delta<1$. Suppose that $W_1, W_2$ are nonnegative and solve
\begin{equation}\label{3.1}
\begin{cases}
W_1'(t)=-a_1W_1^{p-1}(t)+b_1W_2^m(t),\ t\in(0,T),\\
W_2'(t)=-a_2W_2^{q-1}(t)+b_2W_1^n(t),\ t\in(0,T).
\end{cases}
\end{equation}
If $(W_1(0),W_2(0))\in \mathcal{Q}$, then $(W_1,W_2) \in \mathcal{Q}$.
\end{lemma}

\begin{lemma}\label{le3.2}
\cite{Chen-Y13}\ Let the hypothesis as in Lemma \ref{le3.1}. Then every nonnegative solution of (\ref{3.1}) vanishes in finite time for every $(W_1(0),W_2(0)) \in \mathcal{Q}$.
\end{lemma}

The following corollary is a direct consequence of Lemma \ref{le3.2} and the comparison argument.

\begin{corollary}\label{co3.1}
Let $a_{i}, b_{i}(i=1,2), m, n$ be positive constants, $1<p,q <2$ and $mn\geq(p-1)(q-1)$.
Assume that $(W_1,W_2)$ satisfies the following differential inequalities
\begin{equation}\label{3.2}
\begin{cases}
W_1'(t)\leq-a_1W_1^{p-1}(t)+b_1W_2^m(t),\\
W_2'(t)\leq-a_2W_2^{q-1}(t)+b_2W_1^n(t).
\end{cases}
\end{equation}
Then every nonnegative solution of (\ref{3.2}) vanishes in finite time for every $(W_1(0),W_2(0)) \in \mathcal{Q}$.
\end{corollary}

The following theorem shows that any solution of (\ref{1.1}) vanishes in finite time
when the nonlinear sources are in some sense weak and when the initial data are ``comparable".

\begin{theorem}\label{thm3.1}
Assume that $(p-1)(q-1)<mn$.

(I)\ If $mn\leq1$ and the initial data $(u_0,v_0)$ satisfy, for some $0<\delta_1<1$, that
\begin{equation}\label{3.3}
\Big(\frac{b_1}{\delta_1a_1}\Big)^{\frac{1}{p-1}}\|v_0\|_{L^r(\Omega)}^{\frac{m}{p-1}}\leq \|u_0\|_{L^s(\Omega)}\leq\Big(\frac{\delta_1a_2}{b_2}\Big)^{\frac{1}{n}}\|v_0\|_{L^r(\Omega)}^{\frac{q-1}{n}},
\end{equation}
then every solution of (\ref{1.1}) vanishes in finite time;

(II)\ If $mn>1$ and the initial data $(u_0,v_0)$ satisfy, for some $0<\delta_2<1$, that
\begin{equation}\label{3.21}
\Big(\frac{b'_1}{\delta_2a'_1}\Big)^{\frac{1}{p-1}}\|v_0\|_{L^{r'}(\Omega)}^{\frac{m_1}{p-1}}\leq
\|u_0\|_{L^{s'}(\Omega)}\leq\Big(\frac{\delta_2a'_2}{b'_2}\Big)^{\frac{1}{n_1}}\|v_0\|_{L^{r'}(\Omega)}^{\frac{q-1}{n_1}},
\end{equation}
then every solution of (\ref{1.1}) vanishes in finite time for sufficiently small
initial data. Here $a_i, b_i, a'_i,b'_i>0(i=1,2)$, $s,r,s',r'>1$, $0<m_1<m$ and $0<n_1\leq n$ are constants to be defined in the process of the proof.
\end{theorem}

\begin{proof}
As a matter of convenience, in what follows, we might as well assume that the weak solution is
appropriately smooth, or else, we can consider the corresponding regularized problem and through
an approximate process, the same result can also be obtained.

{\bf Case I: $mn\leq1$.} In this case, there exist constants $r,s>1$ such that $m\leq\frac{r}{s}\leq\frac{1}{n}$.
Multiplying the first equation of (\ref{1.1}) by $u^{s-1}$, the second equation by $v^{r-1}$ and
integrating the results over $\Omega$, one obtains

\begin{equation}\label{3.4}
\frac{1}{s}\frac{d}{dt}\int_\Omega
u^{s}dx+\frac{(s-1)p^p}{(s+p-2)^p}\int_\Omega|\nabla
u^{\frac{s+p-2}{p}}|^pdx=\int_\Omega v^{m}u^{s-1}dx,
\end{equation}

\begin{equation}\label{3.5}
\frac{1}{r}\frac{d}{dt}\int_\Omega
v^{r}dx+\frac{(r-1)q^q}{(r+q-2)^q}\int_\Omega|\nabla
v^{\frac{r+q-2}{q}}|^qdx=\int_\Omega u^{n}v^{r-1}dx.
\end{equation}
The proof of this case will be divided into two subcases.

Subcase 1: $N\geq2$. Since $p<N$, by choosing $s\geq\frac{N(2-p)}{p}$ (which implies $\frac{s+p-2}{p}\frac{Np}{N-p}\geq s$) and recalling
Sobolev embedding theorem \Big($W_0^{1,p}(\Omega)\hookrightarrow L^{p^*}(\Omega)(p^*=\frac{Np}{N-p})$\Big) and H\"{o}lder's inequality, we have

$$\|u\|_s^{\frac{s+p-2}{p}}\leq|\Omega|^{\frac{s+p-2}{sp}-\frac{1}{p^*}}\|u^{\frac{s+p-2}{p}}\|_{p^*}\leq|\Omega|^{\frac{s+p-2}{sp}-\frac{1}{p^*}}
\gamma_1\|\nabla u^{\frac{s+p-2}{p}}\|_p,$$
$$\int_{\Omega}v^mu^{s-1}\leq\|u\|_s^{s-1}\|v\|_{ms}^m\leq|\Omega|^{\frac{1}{s}-\frac{m}{r}}\|u\|_s^{s-1}\|v\|_{r}^m,$$
where $\gamma_1>0$ is the embedding constant. Substituting the above two inequalities into (\ref{3.4}) yields
\begin{equation}\label{3.6}
\frac{1}{s}\frac{d}{dt}\int_\Omega
u^{s}dx\leq-\frac{(s-1)p^p}{(s+p-2)^p}\gamma_1^{-p}|\Omega|^{-p(\frac{s+p-2}{sp}-\frac{1}{p^*})}
\|u\|_s^{s+p-2}+|\Omega|^{\frac{1}{s}-\frac{m}{r}}\|u\|_s^{s-1}\|v\|_{r}^m.
\end{equation}
Set $J_1(t)=\int_\Omega u^s(x,t)dx$, $J_2(t)=\int_\Omega v^r(x,t)dx$. Then (\ref{3.6}) can be rewritten as
\begin{equation}\label{3.7}
\frac{1}{s}J_1'(t)\leq-\frac{(s-1)p^p}{(s+p-2)^p}\gamma_1^{-p}|\Omega|^{-p(\frac{s+p-2}{sp}-\frac{1}{p^*})}
J_1^{\frac{s+p-2}{s}}+|\Omega|^{\frac{1}{s}-\frac{m}{r}}J_1^{\frac{s-1}{s}}J_{2}^{\frac{m}{r}}.
\end{equation}
Symmetrically, we have

\begin{equation}\label{3.8}
\frac{1}{r}J_2'(t)\leq-\frac{(r-1)q^q}{(r+q-2)^q}\gamma_2^{-q}|\Omega|^{-q(\frac{r+q-2}{rq}-\frac{1}{q^*})}
J_2^{\frac{r+q-2}{r}}+|\Omega|^{\frac{1}{r}-\frac{n}{s}}J_2^{\frac{r-1}{r}}J_{1}^{\frac{n}{s}},
\end{equation}
where $r>\max\{1,\frac{N(2-q)}{q}\}$, $q^*=\frac{Nq}{N-q}$ and $\gamma_2>0$ is the embedding constant. Set
\begin{equation*}
\begin{split}
&W_1(t)=J_1^{\frac{1}{s}}(t),\qquad \qquad \ W_2(t)=J_2^{\frac{1}{r}}(t),\\
&a_1=\frac{(s-1)p^p}{(s+p-2)^p}\gamma_1^{-p}|\Omega|^{-p(\frac{s+p-2}{sp}-\frac{1}{p^*})},\ \
b_1=|\Omega|^{\frac{1}{s}-\frac{m}{r}},\\
&a_2=\frac{(r-1)q^q}{(r+q-2)^q}\gamma_2^{-q}|\Omega|^{-q(\frac{r+q-2}{rq}-\frac{1}{q^*})},\ \
b_2=|\Omega|^{\frac{1}{r}-\frac{n}{s}}.
\end{split}
\end{equation*}
Then we can deduce from  (\ref{3.7}) and (\ref{3.8}) that
\begin{equation}\label{3.9}
\begin{cases}
W_1'(t)\leq-a_1W_1^{p-1}(t)+b_1W_2^m(t),\\
W_2'(t)\leq-a_2W_2^{q-1}(t)+b_2W_1^n(t).
\end{cases}
\end{equation}
Recalling (\ref{3.3}) and Corollary \ref{co3.1}, we know $(W_1(t),W_2(t))$ vanishes in finite time, and so does $(u,v)$.

Subcase 2: $N=1$. Since $1<p<2$, by choosing $s\geq2$ (which implies $\frac{2(s+p-2)}{p}\geq s$), and recalling
Sobolev embedding theorem ($W_0^{1,p}(\Omega)\hookrightarrow L^2(\Omega)$) and H\"{o}lder's inequality, we obtain

$$\|u\|_s^{\frac{s+p-2}{p}}\leq|\Omega|^{\frac{s+p-2}{sp}-\frac{1}{2}}\|u^{\frac{s+p-2}{p}}\|_2\leq\gamma_3|\Omega|^{\frac{s+p-2}{sp}-\frac{1}{2}}
\|\nabla u^{\frac{s+p-2}{p}}\|_p,$$
$$\int_{\Omega}v^mu^{s-1}\leq\|u\|_s^{s-1}\|v\|_{ms}^m\leq|\Omega|^{\frac{1}{s}-\frac{m}{r}}\|u\|_s^{s-1}\|v\|_{r}^m,$$
Symmetrically, one has for all $r\geq2$ that
$$\|v\|_r^{\frac{r+q-2}{q}}\leq|\Omega|^{\frac{r+q-2}{rq}-\frac{1}{2}}\|v^{\frac{r+q-2}{q}}\|_2\leq\gamma_4|\Omega|^{\frac{r+q-2}{rq}-\frac{1}{2}}
\|\nabla v^{\frac{r+q-2}{q}}\|_q,$$
$$\int_{\Omega}u^nv^{r-1}\leq\|v\|_r^{r-1}\|u\|_{nr}^n\leq|\Omega|^{\frac{1}{r}-\frac{n}{s}}\|v\|_r^{r-1}\|u\|_{s}^n,$$
Here $\gamma_3,\gamma_4>0$ are the embedding constants. By applying the foregoing arguments we can show that $(u,v)$ vanishes in
finite time.

{\bf Case II: $mn>1$.} Since $mn>(p-1)(q-1)$, there exist constants $l_1,l_2>0$ such that
$\frac{m}{p-1}>\frac{l_1}{l_2}>\frac{q-1}{n}$. For sufficiently small
$k>0$, it is easily verified that $(k^{l_1}\psi_p(x),
k^{l_2}\psi_q(x))$ is a supersolution of (\ref{1.1}) provided that $(u_0(x),
v_0(x))\leq(k^{l_1}\psi_p(x), k^{l_2}\psi_q(x))$, where
$\psi_p(x)$ and $\psi_q(x)$ are the unique positive solutions of the following
two elliptic problems, respectively,

\begin{equation}\label{3.22}
-\mathrm{div}(|\nabla\psi|^{p-2}\nabla\psi)=1,\ x\in\Omega,\ \
\psi(x)=\delta_0>0,\ x\in\partial\Omega;
\end{equation}
and
\begin{equation}\label{3.23}
-\mathrm{div}(|\nabla\psi|^{q-2}\nabla\psi)=1,\ x\in\Omega,\ \
\psi(x)=\delta_0>0,\ x\in\partial\Omega.
\end{equation}
Moreover, $\psi_p(x),\psi_q(x)\geq\delta_0$ for all $x\in\Omega$.
Thus, the application of Lemma \ref{comparison} guarantees that
$$(u(x,t),v(x,t))\leq(k^{l_1}\psi_p(x), k^{l_2}\psi_q(x))\leq (k^{l_1}M_p,k^{l_2}M_q),\qquad\qquad (B)$$
where $M_p=\|\psi_p\|_{L^\infty(\Omega)}$ and
$M_q=\|\psi_q\|_{L^\infty(\Omega)}$. With the help of (B) we obtain from (\ref{3.4})
and (\ref{3.5}) that
\begin{equation}\label{3.24}
\frac{1}{s}\frac{d}{dt}\int_\Omega
u^{s}dx+\frac{(s-1)p^p}{(s+p-2)^p}\int_\Omega|\nabla
u^{\frac{s+p-2}{p}}|^pdx\leq k^{l_2(m-m_1)}M_q^{m-m_1}\int_\Omega
v^{m_1}u^{s-1}dx,
\end{equation}
\begin{equation}\label{3.25}
\frac{1}{r}\frac{d}{dt}\int_\Omega
v^{r}dx+\frac{(r-1)q^q}{(r+q-2)^q}\int_\Omega|\nabla
v^{\frac{r+q-2}{q}}|^qdx\leq k^{l_1(n-n_1)}M_p^{n-n_1}\int_\Omega u^{n_1}v^{r-1}dx,
\end{equation}
where $0<m_1\leq m$, $0<n_1\leq n$ and $(p-1)(q-1)<m_1n_1\leq1$.
The remaining discussion will still be divided into two subcases.
For the subcases $N\geq2$, by applying the arguments similar to those in the proof of Case I
we arrive at
\begin{equation*}
\begin{cases}
W_1'(t)\leq-a'_1W_1^{p-1}(t)+b'_1W_2^{m_1}(t),\\
W_2'(t)\leq-a'_2W_2^{q-1}(t)+b'_2W_1^{n_1}(t),
\end{cases}
\end{equation*}
where $s',r'>1$ satisfying $m_1\leq\frac{r'}{s'}\leq\frac{1}{n_1}$ and
\begin{equation*}
\begin{split}
&W_1(t)=\Big(\int_\Omega u^{s'}(x,t)dx\Big)^{\frac{1}{s'}},\qquad\qquad W_2(t)=\Big(\int_\Omega v^{r'}(x,t)dx\Big)^{\frac{1}{r'}},\\
&a'_1=\frac{(s'-1)p^p}{(s'+p-2)^p}\gamma_1^{-p}|\Omega|^{-p(\frac{s'+p-2}{s'p}-\frac{1}{p^*})},\
\ b'_1=k^{l_2(m-m_1)}M_q^{m-m_1}|\Omega|^{\frac{1}{s'}-\frac{m_1}{r'}},\\
&a'_2=\frac{(r'-1)q^q}{(r'+q-2)^q}\gamma_2^{-q}|\Omega|^{-q(\frac{r'+q-2}{r'q}-\frac{1}{q^*})},\
\ b'_2=k^{l_1(n-n_1)}M_p^{n-n_1}|\Omega|^{\frac{1}{r'}-\frac{n_1}{s'}}.
\end{split}
\end{equation*}
Noticing $m_1n_1>(p-1)(q-1)$ and recalling (\ref{3.21}), we see by applying
Corollary \ref{co3.1} that $(W_1,W_2)$ vanishes in finite time and so does $(u,v)$.
The subcase $N=1$ can be treated similarly whose details are omitted. The proof is complete.
\end{proof}

In order to show whether the solutions of (\ref{1.1}) will vanish in finite time or not
for the case $(p-q)(q-1)=mn$, we first consider the following quasilinear elliptic problems
\begin{equation}\label{3.10}
-\mathrm{div}(|\nabla\varphi|^{p-2}\nabla\varphi)=1,\ x\in\Omega,\ \ \varphi(x)=0,\ x\in\partial\Omega,
\end{equation}
and
\begin{equation}\label{3.11}
-\mathrm{div}(|\nabla\varphi|^{q-2}\nabla\varphi)=1,\ x\in\Omega,\ \ \varphi(x)=0,\ x\in\partial\Omega,
\end{equation}
and denote by $\varphi_{p}$ and $\varphi_{q}$ the unique solutions
of (\ref{3.10}) and (\ref{3.11}),  respectively. It is well known
(and can be deduced by the strong maximum principle \cite{Vazquez84})
that  $\varphi_{p}(x),\varphi_{q}(x)>0$ in $\Omega$.  Moreover, by the
standard De Giorgi iteration process (see \cite{Dibenedetto93}) we know that there exist
positive constants $M_{p}=M_p(\Omega)$, $M_{q}=M_q(\Omega)$
such that $M_{p}=\|\varphi_{p}\|_{L^\infty(\Omega)}\leq C_1|\Omega|^\alpha$ and
$M_{q}=\|\varphi_{q}\|_{L^\infty(\Omega)}\leq C_2|\Omega|^\beta$, where
$C_1,C_2,\alpha,\beta$ are positive constants depending only on $N,p$ and $q$.
In particular, $M_p,M_q\rightarrow 0$ as
$|\Omega|\rightarrow 0$. The comparison principles for (\ref{3.10}) and (\ref{3.11})
also imply that $M_p(\Omega)$ and $M_q(\Omega)$ are monotonic increasing with respect to $\Omega$ in the sense of
set inclusion relation, namely $M_p(\Omega_1)\leq M_p(\Omega_2)$ and $M_q(\Omega_1)\leq M_q(\Omega_2)$ if $\Omega_1\subset \Omega_2$.

\begin{theorem}\label{thm3.3}
Assume that $(p-1)(q-1)=mn$ and $|\Omega|$ is suitably small.
Then there exists a solution of (\ref{1.1}) vanishing in finite time for suitably
small initial data.
\end{theorem}

\begin{proof}
We shall prove this theorem by constructing a proper supersolution. Set
\begin{equation}\label{3.12}
\bar{u}(x,t)=g_1(t)\varphi_{p0}(x),\ \ \bar{v}(x,t)=g_2(t)\varphi_{q0}(x),
\end{equation}
where $g_1(t),g_2(t)$ are two smooth nonincreasing functions to be determined and $\varphi_{p0}$,
$\varphi_{q0}$ are the unique positive solutions of (\ref{3.10}) and (\ref{3.11}) with $\Omega$ replaced by
some smooth domain $\Omega_0$ satisfying $\Omega\subset\subset\Omega_0$, respectively.
Denote $M_{p0}=\|\varphi_{p0}\|_{L^\infty(\Omega_0)}$, $M_{q0}=\|\varphi_{q0}\|_{L^\infty(\Omega_0)}$
and $\sigma=\min\{\min\limits_{x\in\overline{\Omega}}\varphi_{p0}(x),\min\limits_{x\in\overline{\Omega}}\varphi_{q0}(x)\}>0$.

Recalling $g'_1(t),g'_2(t)\leq0$, we can show by direct calculation that $(\overline{u},\overline{v})$ satisfies (in the weak sense)
the following
\begin{eqnarray}\label{3.13}
&&\bar{u}_t-\mathrm{div}(|\nabla\bar{u}|^{p-2}\nabla\bar{u})-\bar{v}^m\nonumber\\
&=&g_1'(t)\varphi_{p0}-g_1^{p-1}(t)\mathrm{div}(|\nabla\varphi_{p0}|^{p-2}\nabla\varphi_{p0})-g_2^m(t)\varphi_{q0}^m\nonumber\\
&=&g_1'(t)\varphi_{p0}+g_1^{p-1}(t)-g_2^m(t)\varphi_{q0}^m\nonumber\\
&\geq&M_{p0}g_1'(t)+g_1^{p-1}(t)-M^m_{q0}g_2^m(t).
\end{eqnarray}
Similarly, we have
\begin{equation}\label{3.14}
\bar{v}_t-\mathrm{div}(|\nabla\bar{v}|^{q-2}\nabla\bar{v})-\bar{u}^n\geq M_{q0}g_2'(t)+g_2^{q-1}(t)-M^n_{p0}g_1^n(t).
\end{equation}
Suppose $\Omega$ is suitably small such that $M_p,M_q<1$. Then by the continuity of the
solutions of Problem (\ref{3.10}) and (\ref{3.11}) with respect to $\Omega$ it is known that
we can choose a suitable smooth domain
$\Omega_0$ fulfilling $\Omega\subset\subset\Omega_0$ such that $M_{p0},M_{q0}<1$.

Let $(g_1(t),g_2(t))$ be the positive solution of the following ordinary
differential equations
\begin{equation}\label{3.15}
\begin{cases}
g_1'(t)=-\frac{1}{M_{p0}}g_1^{p-1}(t)+\frac{M^m_{q0}}{M_{p0}}g_2^m(t),\\
g_2'(t)=-\frac{1}{M_{q0}}g_2^{q-1}(t)+\frac{M^n_{p0}}{M_{q0}}g_1^n(t),\\
\Big(\frac{M^m_{q0}}{\delta}g_2^m(0)\Big)^{\frac{1}{p-1}}\leq g_1(0)\leq\Big(\frac{\delta}{M^n_{p0}}g_2^{q-1}(0)\Big)^{\frac{1}{n}},
\end{cases}
\end{equation}
where $\delta>0$ satisfying $M^m_{q0}M^n_{p0}<\delta^2<1$. By Corollary \ref{co3.1}, we know that $(g_1(t), g_2(t))$ vanishes at some finite time $T_0>0$.

By combining (\ref{3.13}), (\ref{3.14}) with (\ref{3.15}) we know that if $(u_0,\ v_0)$ is sufficiently small such that $u_0(x)\leq g_1(0)\varphi_{p0}(x),\ v_0(x)\leq g_2(0)\varphi_{q0}(x)$ in $\Omega$, then $(\bar{u}, \bar{v})$ defined in (\ref{3.12}) is a supersolution of (\ref{1.1}) which vanishes at $T_0$.
For any fixed $0<T<T_0$, there exist two positive constants $C_1$ and $C_2$ such that $C_1\leq\bar{u},\bar{v}\leq C_2$ on $\overline{\Omega}\times[0,T]$.
Let $(u,v)$ be a solution of Problem (\ref{1.1}), then by the comparison principle (Lemma \ref{comparison}) we know that
$(u(x,t),v(x,t))\leq(\bar{u}(x,t),\bar{v}(x,t))$ for any $(x,t)\in \Omega\times[0,T]$.
By the arbitrariness of $T<T_0$, we see that $u(x,T_1)=v(x,T_1)\equiv0$ for some $T_1\leq T_0$.
If we take $u(x,t)=v(x,t)\equiv0$ for all $t\geq T_1$, then $(u(x,t),v(x,t))$ vanishes at the finite time
$T_1$ and clearly it is a weak solution of Problem (\ref{1.1}). The proof is complete.
\end{proof}

\begin{remark}
The methods used in this paper can also be applied to deal with systems (\ref{1.1}) with nonlocal sources,
that is with $v^m$ and $u^n$ replaced by $\int_\Omega v^m(y,t)dy$ and $\int_\Omega u^n(y,t)dy$, respectively,
and the results are almost the same as the ones obtained above. Interested readers may check it themselves.
\end{remark}

To give some sufficient conditions for the non-extinction of solutions to systems like (\ref{1.1}) is much
more challenging and there is no result except some partial answer obtained in \cite{Friedman92} for a system of
semilinear parabolic variational inequalities. In the last part of this section,
we will derive some non-extinction criteria for solutions to (\ref{1.1}) in some special cases.
Our result shows that when the nonlinear sources are in some sense strong,
Problem (\ref{1.1}) admits at least one non-extinction solution for any positive smooth initial data.

\begin{theorem}\label{thm3.3}
Assume that $1<p=q<2$, $0<m,n\leq p-1$ and $mn<(p-1)^2$. Then Problem (\ref{1.1})
admits at least one non-extinction solution for any smooth positive initial data $(u_0,v_0)$.
\end{theorem}

\begin{proof}
We will prove this theorem by constructing a pair of ordered super and subsolution and utilizing
the monotonic iteration process. The whole process is divided into four steps.

{\bf Step 1.} We first construct a non-extinction subsolution of (\ref{1.1}). For this,
denote by $\lambda_1>0$ be the first eigenvalue of the following eigenvalue problem
\begin{equation}\label{3.16}
-\mathrm{div}(|\nabla\phi|^{p-2}\nabla\phi)=\lambda|\phi(x)|^{p-2}\phi(x),\ x\in\Omega,\ \ \phi(x)=0,\ x\in\partial\Omega,
\end{equation}
and by $\phi_1(x)$ the first eigenfunction. We may choose $\phi_1(x)>0$ in $\Omega$ and normalize it with $\|\phi_1\|_{L^\infty(\Omega)}=1$.

Since $mn<(p-1)^2$, there exists two positive constants $\theta_1,\theta_2$ such that
\begin{equation}\label{3.17}
\frac{m}{p-1}<\frac{\theta_1}{\theta_2}<\frac{p-1}{n}.
\end{equation}
Define $\underline{u}=k^{\theta_1}\phi_1(x)$, $\underline{v}=k^{\theta_2}\phi_1(x)$. Recalling $0<m,n\leq p-1$, by direct computation we see that
$(\underline{u},\underline{v})$ satisfies (in the weak sense)
\begin{eqnarray}\label{3.18}
&&\underline{u}_t-\mathrm{div}(|\nabla \underline{u}|^{p-2}\nabla \underline{u})-\underline{v}^m
=\lambda_1k^{\theta_1(p-1)}\phi_1^{p-1}-k^{m\theta_2}\phi_1^{m}\nonumber\\
&=&\phi_1^{m}(\lambda_1k^{\theta_1(p-1)}\phi_1^{p-1-m}-k^{m\theta_2})\leq\phi_1^{m}(\lambda_1k^{\theta_1(p-1)}-k^{m\theta_2}),
\end{eqnarray}
and
\begin{eqnarray}\label{3.19}
&&\underline{v}_t-\mathrm{div}(|\nabla \underline{v}|^{p-2}\nabla \underline{v})-\underline{u}^n
=\lambda_1k^{\theta_2(p-1)}\phi_1^{p-1}-k^{n\theta_1}\phi_1^{n}\nonumber\\
&=&\phi_1^{n}(\lambda_1k^{\theta_2(p-1)}\phi_1^{p-1-n}-k^{n\theta_1})\leq\phi_1^{n}(\lambda_1k^{\theta_2(p-1)}-k^{n\theta_1}).
\end{eqnarray}
Combining (\ref{3.17}), (\ref{3.18}) with (\ref{3.19}) we know that there exists a constant $k_1>0$ such that
for all $k\in(0,k_1]$, the following relations hold
\begin{equation}\label{3.20}
\begin{cases}
\underline{u}_t-\mathrm{div}(|\nabla \underline{u}|^{p-2}\nabla \underline{u})-\underline{v}^m\leq0,&x\in\Omega,\ t>0,\\
\underline{v}_t-\mathrm{div}(|\nabla \underline{v}|^{p-2}\nabla \underline{v})-\underline{u}^n\leq0,&x\in\Omega,\ t>0.
\end{cases}
\end{equation}
On the other hand, for any $u_0,v_0\in C^1(\overline{\Omega})$ satisfying $u_0,v_0>0$ in $\Omega$, $u_0,v_0=0$,
$\frac{\partial u_0}{\partial\nu}<0$, $\frac{\partial v_0}{\partial\nu}<0$ on
$\partial\Omega$, there exists a constant $k_2>0$ such that for all $k\in(0,k_2]$ we have
\begin{equation}\label{3.21}
u_0(x)\geq k^{\theta_1}\phi_1(x),\ \ v_0(x)\geq k^{\theta_2}\phi_1(x),\ \ x\in\Omega.
\end{equation}
From (\ref{3.20}) and (\ref{3.21}) we know that $(\underline{u},\underline{v})$ is a non-extinction weak subsolution
of (\ref{1.1}) for all $0<k\leq\min\{k_1,k_2\}$.

{\bf Step 2.} To construct a supsolution of (\ref{1.1}), let us consider the following auxiliary system
\begin{equation}\label{3.22}
\begin{cases}
u_t=\mathrm{div}(|\nabla u|^{p-2}\nabla u)+(v_++1)^m,&\ x\in\Omega,\ t>0,\\
v_t=\mathrm{div}(|\nabla v|^{p-2}\nabla v)+(u_++1)^n,&\ x\in\Omega,\ t>0,\\
u(x,t)=v(x,t)=0,&x\in\partial\Omega,\ t>0,\\
u(x,0)=u_0(x),\ v(x,0)=v_0(x),&x\in\Omega.
\end{cases}
\end{equation}
By applying the arguments
similar to those in the proof of Proposition \ref{existence} we know that Problem (\ref{3.22})
admits a weak solution $(\overline{u},\overline{v})$. By the weak maximum principle it is known
that $(\overline{u},\overline{v})$ is nonnegative. Moreover, $(\overline{u},\overline{v})$
exists globally and is locally bounded if $mn\leq1$. If we can show that $(\underline{u},\underline{v})\leq(\overline{u},\overline{v})$,
then there exists a solution $(u,v)$ of (\ref{1.1}) satisfying $(\underline{u},\underline{v})\leq(u,v)\leq(\overline{u},\overline{v})$.

{\bf Step 3.} Fix $T\in(0,\infty)$. From the definition of weak super and subsolutions, we obtain, for any $0\leq\phi_1\in E_{p0}$ and $0\leq\phi_2\in E_{q0}$,
\begin{eqnarray*}
&&\iint_{Q_T}\Big(\frac{\partial \underline{u}}{\partial t}-\frac{\partial \overline{u}}{\partial t}\Big)\phi_1dxd\tau
+\iint_{Q_T}(|\nabla \underline{u}|^{p-2}\nabla \underline{u}-|\nabla \overline{u}|^{p-2}\nabla \overline{u})\nabla\phi_1dxd\tau\\
&&\leq\iint_{Q_T}(\underline{v}^m-(\overline{v}_++1)^m)\phi_1dxd\tau,\\
&&\iint_{Q_T}\Big(\frac{\partial \underline{v}}{\partial t}-\frac{\partial \overline{v}}{\partial t}\Big)\phi_2dxd\tau
+\iint_{Q_T}(|\nabla \underline{v}|^{p-2}\nabla \underline{v}-|\nabla \overline{v}|^{p-2}\nabla \overline{v})\nabla\phi_2dxd\tau\\
&&\leq\iint_{Q_T}(\underline{u}^n-(\overline{u}_++1)^n)\phi_2dxd\tau.
\end{eqnarray*}
Noticing $0<m,n<1$ and choosing $\phi_1=\chi_{[0,t]}(\underline{u}-\overline{u})_+$, $\phi_2=\chi_{[0,t]}(\underline{v}-\overline{v})_+$ for any $t\in(0,T)$,
we have
\begin{eqnarray*}
&&\iint_{Q_t}\Big(\frac{\partial \underline{u}}{\partial t}-\frac{\partial \overline{u}}{\partial t}\Big)(\underline{u}-\overline{u})_+dxd\tau
+\iint_{Q_t}(|\nabla \underline{u}|^{p-2}\nabla \underline{u}-|\nabla \overline{u}|^{p-2}\nabla \overline{u})\nabla(\underline{u}-\overline{u})_+dxd\tau\\
&&\leq\iint_{Q_t}(\underline{v}^m-(\overline{v}_++1)^m)(\underline{u}-\overline{u})_+dxd\tau\leq m\iint_{Q_t}(\underline{v}-\overline{v})_+(\underline{u}-\overline{u})_+dxd\tau,
\end{eqnarray*}
which implies
\begin{eqnarray*}
&&\int_{\Omega}(\underline{u}-\overline{u})^2_+dx+2\iint_{Q_t}(|\nabla \underline{u}|^{p-2}\nabla \underline{u}-|\nabla \overline{u}|^{p-2}\nabla \overline{u})\nabla(\underline{u}-\overline{u})_+dxd\tau\nonumber\\
&\leq& 2m\iint_{Q_t}(\underline{v}-\overline{v})_+(\underline{u}-\overline{u})_+dxd\tau.
\end{eqnarray*}
Symmetrically, we have
\begin{eqnarray*}
&&\int_{\Omega}(\underline{v}-\overline{v})^2_+dx+2\iint_{Q_t}(|\nabla \underline{v}|^{p-2}\nabla \underline{v}-|\nabla \overline{v}|^{p-2}\nabla \overline{v})\nabla(\underline{v}-\overline{v})_+dxd\tau\nonumber\\
&\leq& 2n\iint_{Q_t}(\underline{v}-\overline{v})_+(\underline{u}-\overline{u})_+dxd\tau.
\end{eqnarray*}
By the monotonicity of $p$-Laplace operator and Gronwall's inequality we have $(\underline{u},\underline{v})\leq(\overline{u},\overline{v})$.

{\bf Step 4.} Define $(u_1,v_1)=(\underline{u},\underline{v})$ and $\{(u_k,v_k)\}_{k\geq2}$
iteratively to be a solution of the following problem
\begin{equation}\label{3.23}
\begin{cases}
u_{kt}=\mathrm{div}(|\nabla u_k|^{p-2}\nabla u_k)+v_{k-1}^m,&\ x\in\Omega,\ t>0,\\
v_{kt}=\mathrm{div}(|\nabla v_k|^{p-2}\nabla v_k)+u_{k-1}^n,&\ x\in\Omega,\ t>0,\\
u(x,t)=v(x,t)=0,&x\in\partial\Omega,\ t>0,\\
u(x,0)=u_0(x),\ v(x,0)=v_0(x),&x\in\Omega.
\end{cases}
\end{equation}
By induction we can prove that $(u_k,v_k)\leq(u_{k+1},v_{k+1})$ and $(u_k,v_k)\leq(\overline{u},\overline{v})$ for all $k\geq1$.
Thus the limits $u(x,t)=\lim_{k\rightarrow\infty}u_k(x,t)$ and $v(x,t)=\lim_{k\rightarrow\infty}v_k(x,t)$ exist for every $x\in\Omega$
and $t>0$ and it is not hard to show that $(u,v)$ is a weak solution of (\ref{1.1}) by the regularities of $\{(u_k,v_k)\}_{k\geq2}$.
Therefore, $(u,v)$ is a non-extinction solution of (\ref{1.1}) since $(u,v)\geq(\underline{u},\underline{v})$. The proof is complete.
\end{proof}

\begin{remark}
Implied by the results of scalar problems (see \cite{Tian08,Yin07}) we conjecture that
(\ref{1.1}) should admit at least one non-extinction solution for any nonnegative initial data $(u_0,v_0)$
when $1<p,q<2$ and $mn<(p-1)(q-1)$.
\end{remark}

\end{document}